\documentclass[12pt]{amsart}
\usepackage{amsmath,latexsym,amsfonts,amssymb,amsthm}
\usepackage{geometry}
\usepackage{hyperref}
\usepackage{mathrsfs}
\usepackage{graphicx,color}
\usepackage[alphabetic]{amsrefs}
\usepackage[all,cmtip]{xy}
\usepackage{bbm, stmaryrd}
\usepackage{tabu}
\usepackage{enumitem}
\usepackage{tikz}
\usepackage{tikz-cd}
\usepackage{comment}

\newtheorem{prop}{Proposition}[section]
\newtheorem{thm}[prop]{Theorem}

\newtheorem{conj}[prop]{Conjecture}
\newtheorem{lem}[prop]{Lemma}

\theoremstyle{definition}
\newtheorem{que}[prop]{Question}
\newtheorem{say}[prop]{}
\newtheorem{defn}[prop]{Definition}
\newtheorem{expl}[prop]{Example}

\newtheorem*{claim*}{Claim}

\newcommand{\bP}{\mathbb{P}}

\newcommand{\bR}{\mathbb{R}}
\newcommand{\bA}{\mathbb{A}}
\newcommand{\bQ}{\mathbb{Q}}

\newcommand{\bN}{\mathbb{N}}

\newcommand{\bG}{\mathbb{G}}

\newcommand{\bm}{\mathfrak{m}}

\newcommand{\cX}{\mathcal{X}}
\newcommand{\cY}{\mathcal{Y}}

\newcommand{\cO}{\mathcal{O}}

\newcommand{\cE}{\mathcal{E}}
\newcommand{\cD}{\mathcal{D}}

\newcommand{\fa}{\mathfrak{a}}

\newcommand{\Aut}{\mathrm{Aut}}

\newcommand{\ord}{\mathrm{ord}}
\newcommand{\Val}{\mathrm{Val}}

\newcommand{\QM}{\mathrm{QM}}

\newcommand{\hvol}{\widehat{\rm vol}}

\newcommand{\fX}{\mathfrak{X}}
\newcommand{\fD}{\mathfrak{D}}
\newcommand{\fE}{\mathfrak{E}}

\newcommand{\cU}{\mathcal{U}}

\newcommand{\oH}{\overline{H}}

\numberwithin{equation}{section}

\title{A note on Koll\'ar valuations}

\date{}

\author{Yuchen Liu}
\address{Department of Mathematics, Northwestern University, Evanston, IL 60208, USA.}
\email{yuchenl@northwestern.edu}

\author{Chenyang Xu}
\address{Department of Mathematics, Princeton University, Princeton, NJ 08544, USA}
\email     {chenyang@princeton.edu}

\begin{document}

\maketitle

\begin{center}
In memory of Gang Xiao
\end{center}

\begin{abstract}
We prove the set of Koll\'ar valuations in the dual complex of a klt singularity with a fixed complement is path connected. We also classify the case when the dual complex is one dimensional. 
\end{abstract}

\section{Introduction}
One new question with a birational geometric feature,  arising from algebraic K-stability theory is to ask what kind of quasi-monomial valuation $v$ over a polarized projective varieties $(X,L)$ satisfies that 
${\rm Gr}_v\bigoplus_{m\in \bN} H^0(X,mL)$ is finitely generated. For now the only main case we have some knowledge is in the Fano setting, i.e.

\begin{que}[Global version]
Let $(X,\Delta)$ be a klt log Fano pair, and $v$ is an lc place of a $\mathbb Q$-complement. Let $r$ such that $r(K_X+\Delta)$ is Cartier. Then what condition implies that 
\[
\bigoplus_{m\in r\cdot \bN} {\rm Gr}_v H^0(X,-m(K_X+\Delta))
\]
is finitely generated? 
\end{que}

There is a local version which implies the global version by taking the cone.

\begin{que}[Local version]
Let $(X={\rm Spec}(R),\Delta)$ be a klt singularity, and $v\in \Val_{X,x}$ is an lc place of a $\mathbb Q$-complement. Then what condition implies that 
$
 {\rm Gr}_v R
$
is finitely generated? 
\end{que}

When $v$ is a divisorial valuation, i.e. ${v}=\ord_E$, then it follows from \cite{BCHM} and our assumption that $E$ is an lc place of a $\mathbb Q$-complement that ${\rm Gr}_E R$ is finitely generated. However, for $v$ with rational rank ${\rm rank}_{\bQ}(v)\ge 2$, the question is quite unclear. Built on \cite{LXZ-HRFG}, in \cite{XZ-HRFG}, a smaller class is sorted out.
\begin{defn}
Let $x\in(X,\Delta)$ be a klt singularity. We say $v\in \Val_{X,x}$ is a \emph{Koll\'ar valuation} if there exists a birational model $\mu\colon Y\to X$ with ${\rm Ex}(\mu)=\sum^p_{i=1}E_i$
and $D\ge 0$ on $Y$ such that 
\begin{enumerate}
\item $(Y,E+D)$ is q-dlt with $\lfloor E+D\rfloor=E$ and $v\in {\rm QM}(Y,E)$, 
\item $K_Y+E+D\ge \mu^*(K_X+\Delta)$ and $-K_Y-E-D$ is ample over $X$. 
\end{enumerate}
We say that $(Y,E+D)$ is a \emph{Koll\'ar model} over $X$, extracting $E_1,\dots,E_p$.
\end{defn}
We have 
\begin{thm}[\cite{XZ-HRFG}]
If $v$ is a Koll\'ar valuation, then ${\rm Gr}_vR$ is finitely generated.
\end{thm}
In fact, it is proved that $v$ is a Koll\'ar valuation if and only if the degeneration  of $(X,\Delta)$ to $X_0={\rm Spec}({\rm Gr}_vR)$ yields a klt pair. Fix a $\bQ$-complement $D$ such that $x$ is the only lc center. It is known that the set of lc places 
\[
{\rm LCP}(X,\Delta+D)=\left\{\mbox{nontrivial valuations }v \mbox{ such that }A_{X,\Delta}(v)=v(D)\right\} 
\]
is a cone over the dual complex $\cD(X,\Delta+D)$, which is a collapsible pseudo-manifold with boundary (see \cite{KK-lc,dFKX-dualcomplex}). Based on the discussion above, it is also worth looking at the set of  all Koll\'ar valuations ${\rm KV}(X,\Delta+D)\subseteq {\rm LCP}(X,\Delta+D)$, which is a cone over a space, denoted by $\cD^{\rm KV}(X,\Delta+D)$. In our note, we will identify a valuation with its non-zero rescaling, and in this sense we talk about a valuation in the dual complex. 

\begin{que}\label{que-kv}
Let $x\in (X,\Delta)$ be a klt singularity, and $D$ a $\bQ$-complement such that $x$ is the only lc center, i.e.  $D$  is an effective $\mathbb Q$-divisor such that $(X,\Delta+D)$ is klt outside $x$, and lc but not klt at $x$. How does $\cD^{\rm KV}(X,\Delta+D)\subseteq \cD(X,\Delta+D)$ look like?
\end{que}
In general, we know little about $\cD^{\rm KV}(X,\Delta+D)$. The following summarizes what has been proved.
\begin{thm}\label{thm-propertyofKollar}
We know
\begin{enumerate}
\item  (\cite{Xu-kollar}) $\cD^{\rm KV}(X,\Delta+D)\neq \emptyset$.
\item (\cite{LX-SDCII})  We fix a log resolution of $(X,\Delta+D)$, which yields a triangulation of $\cD(X,\Delta+D)$. If a point $x\in \cD^{\rm KV}(X,\Delta+D)$, then there exists a neighborhood $U$ of $x$ in the smallest affine linear subspace $V$ defined over $\bQ$  containing $x$, such that $U\subseteq  \cD^{\rm KV}(X,\Delta+D)$. 
\end{enumerate}
\end{thm}
One easily see the conclusion in (2) does not depend on the choice of the log resolution.  

Our first main result proves  the path connectedness of the locus of Koll\'ar valuations.

\begin{thm}\label{thm-connected} Let $x\in (X,\Delta)$ be a klt singularity, and let $D$ be a $\bQ$-complement such that $\{x\}$ is the only  lc center of $(X,\Delta+D)$. Then 
$\cD^{\rm KV}(X,\Delta+D)$ is path connected.
\end{thm}

We make the following conjectures.

\begin{conj}\label{conj-kollar valuations}
We conjecture $\cD^{\rm KV}(X,\Delta+D)$ satisfies that
\begin{enumerate}
\item (Local closedness, weak version) There is a rational triangulation of  $\cD(X,\Delta+D)$ such that for each open simplex $C^{\circ}$, $\cD^{\rm KV}(X,\Delta+D)\cap C^{\circ}$ is open in its closure $\overline{\cD^{\rm KV}(X,\Delta+D)\cap C^{\circ}}\subset C^{\circ}$.
\item[(1')] (Local closedness, strong version) There is a rational triangulation of  $\cD(X,\Delta+D)$ such that for each open simplex $C^{\circ}$, $\cD^{\rm KV}(X,\Delta+D)\cap C^{\circ} \subset C^{\circ}$ is open.
\item (Finiteness, weak version) If $\cD(X,\Delta+D)=\cD^{\rm KV}(X,\Delta+D)$, i.e. $\cD(X,\Delta+D)$ consists of Koll\'ar valuations, then there is a finite triangulation of $\cD(X,\Delta+D)$, such that each simplex can be realized on a Koll\'ar model. 
\item[(2')] (Finiteness, strong version) If $C\subseteq \cD^{\rm KV}(X,\Delta+D)$ is a closed cell of  $\cD(X,\Delta+D)$ after a triangulation given by a log resolution, then there is a finite triangulation of $C=\cup^N_{i=1} C_i$, such that each $C_i$ is realized in a Koll\'ar model. 
 \end{enumerate}
\end{conj}

See Example \ref{ex-conecubic} for sharpness of our formulation. 

In Section \ref{s-dual complex dim 1}, we study this conjecture when the dual complex is one dimensional, and completely address the question in this case.

\begin{thm}\label{thm-nghbKollar}
Assume $\cD(X,\Delta+D)$ is homeomorphic to $[0,1]$. We denote by $v_t$ the valuation corresponding to $t\in [0,1]$. The set of Koll\'ar valuations $\cD^{\rm KV}(X,\Delta+D)$ is
\begin{enumerate}
\item either precisely one of $E_0$ or $E_1$; 
\item or $v_t$ for all $t\in [0,1]$.
\end{enumerate}
\end{thm}

\medskip

\noindent{\bf Notation and Conventions}: We follow the notation and conventions in \cites{KollarMori, Kollar13, Xu-Kbook}. By a singularity $x\in X$, we mean $X={\rm Spec}(R)$ for a local ring $R$ which is essentially of finite type, $x$ is the closed point. If $v$ is a valuation on $K(X)$, whose center on $X$ is $x$, then we denote by ${\rm Gr}_vR$ the associated graded ring, i.e. 
${\rm Gr}_vR=\bigoplus_{\lambda \in \bR}\fa_{\ge \lambda}/\fa_{>\lambda}$, where 
\[
\fa_{\ge \lambda} (\mbox{resp. } \fa_{> \lambda})=\left\{f\in R\mid v(f) \ge  (\mbox{resp. } >)\lambda \right\} \, .
\]
When $v=\ord_E$ for some divisor $E$ over $X$, we also write ${\rm Gr}_ER$ for ${\rm Gr}_{\ord_E}R$. Let $(Y,E)$ be a log smooth model or more generally a toroidal model such that components of $E$ are $\bQ$-Cartier, we denote by $\QM(Y,E)$ the set of quasi-monomial valuations (see e.g. \cite[Definition 2.8]{XZ-HRFG}). A qdlt pair is defined as \cite[Definition 35]{dFKX-dualcomplex}.
\medskip

\noindent{\bf Acknowledgement}: We want to thank Harold Blum and Ziquan Zhuang for many discussions. We are also grateful to the referees for helpful comments. YL is partially supported by  NSF CAREER Grant DMS-2237139 and the Alfred P. Sloan Foundation. CX is partially supported by Simons Investigator and NSF Grant DMS-2201349.
\section{Connectedness}

In this section, we aim to prove Theorem \ref{thm-connected}.

\begin{lem}\label{lem-ACC}
Fix  a klt singularity $x\in (X,\Delta)$ and a $\bQ$-complement $D$ such that $(X,\Delta+D)$ is klt outside $x$. There exists a positive $\varepsilon$ depending only on $\dim(X)$, ${\rm Coeff}(\Delta)$ and ${\rm Coeff}(D)$, such that any lc place $E$ of a lc pair $(X,\Delta+(1-\varepsilon)D+G)$ for an effective $\bQ$-divisor $G$  belongs to $\cD(X,\Delta+D)$.
\end{lem}
\begin{proof}
This follows directly from \cite[Lemma 5.5]{Zhu22} by applying ACC of log canonical thresholds \cite{HMX-ACC} (see also \cite[Proof of Proposition 6.9]{LLX-Tiansurvey}). 
\end{proof}
\begin{proof}[Proof of Theorem \ref{thm-connected}]
Replacing $x\in (X,\Delta)$ and $D$ by $x\in (X,\Delta+(1-\varepsilon)D)$ and $\varepsilon D$ respectively  as in Lemma \ref{lem-ACC}, we may assume that lc places of any $\bQ$-complement of $(X, \Delta)$ are contained in $\cD(X,\Delta+D)$.
Using Theorem \ref{thm-propertyofKollar}, we can replace two Koll\'ar valuations by nearby Koll\'ar components. Thus we may assume two valuations are Koll\'ar components $E_0$ and $E_1$. Let $\mu_i\colon Y_i\to X$ $(i=0,1)$ be the model extracting the Koll\'ar component $E_i$, and $H_i$ the pushforward of a general effective $\bQ$-divisor in $|-K_{Y_i}-E_i-\mu^{-1}_{i*}\Delta|_{\bQ}$. So $H_i$ is a $\bQ$-complement such that $E_i$ are the only lc places of $(X,\Delta+H_i)$. Applying Lemma \ref{lem-ACC} to $E_i$, there exists  $\varepsilon_i $ such that $E_i$ are the only lc places of any $\bQ$-complement of $(X,\Delta+(1-\varepsilon_i)H_i)$. In particular, $E_i$ is the minimizer of $\hvol_{X,\Delta+(1-\varepsilon_i)H_i}$.

For a real number $t\in [0,1]$, define $\oH_t:= (1-t)(1-\varepsilon_0)H_0+t(1-\varepsilon_1)H_1$. Thus $x\in (X, \Delta +\oH_t)$ is a klt singularity with $\bR$-boundary divisors.
Let $v_t$ be a minimizer of $\hvol_{X,\Delta+\oH_t}(\cdot)$ such that $A_{X,\Delta}(v_t) = 1$, whose existence and uniqueness follows from  \cite{Blu-existence, Xu-quasimonomial, XZ-unique} and their generalization to $\bR$-divisors \cite[Theorems 3.3 and 3.4]{HLQ}. Moreover, by \cite{XZ-HRFG} and its generalization to $\bR$-divisors \cite[Theorem 2.19]{Zhu23}, we know that $v_t$ is a Koll\'ar valuation over $x\in (X, \Delta+\oH_t)$ (hence over $x\in (X,\Delta)$). By \cite[Theorem 2.20]{HLQ} (see also \cite[Theorem 1.3]{LX-SDCI}),  there exists a sequence of Koll\'ar components $(S_{t,j})_{j\in \bN}$ over $x\in (X, \Delta+\oH_t)$ such that 
$v_t = \lim_{j\to \infty} \frac{\ord_{S_{t,j}}}{A_{X,\Delta}(S_{t,j})}$. Hence $S_{t,j}$ is a Koll\'ar component over $x\in (X, \Delta)$ which implies that it is an lc place of a $\bQ$-complement of $(X,\Delta)$. By Lemma \ref{lem-ACC} we know that $S_{t,j} \in \cD(X,\Delta+ D)$, which implies that their rescaled limit $v_t\in \cD(X,\Delta + D)$. 
Moreover, since 
$(v,t)\mapsto\hvol_{X,\Delta+\oH_t}(v)$ 
is a continuous function on $\cD(X,\Delta+D)\times [0,1]$,
the function 
\[
[0,1]\to \cD(X,\Delta+D), \ \ \ t\mapsto v_t
\]
is continuous by Lemma \ref{lem-continuousminimizer}.
\end{proof}
\begin{lem}\label{lem-continuousminimizer}
Let $f:W\times [0,1]\to \bR$ be a continuous function where $W$ is a compact topological space. Assume that for every $t\in [0,1]$, there is a unique minimizer $v_t\in W$ of the function $f_t$ which is the restriction of $f$ on $W\times t$. Then $t\mapsto v_t$ is a continuous function from $[0,1]$ to $W$. 
\end{lem}

\begin{proof}
Consider the subset $\cU\subset W\times [0,1]$ defined by 
\[
\cU:=\{(x,t)\in W\times [0,1]\mid \textrm{there exists } y\in W\textrm{ such that }f(x,t)>f(y,t)\}.
\]
We claim that $\cU$ is open in $W\times [0,1]$. Let $(x,t)\in \cU$ be an arbitrary point. Then there exists $y\in W$ such that $f(x,t)>f(y,t)$. Since $f$ is continuous, for $\varepsilon=\frac{f(x,t)-f(y,t)}{2}>0$ there exists open neighborhoods $x\in U_x\subset W$, $y\in U_y\subset W$, and $t\in U_t\subset [0,1]$, such that $f(x',t') > f(x,t) - \varepsilon$ and $f(y',t')<f(y,t)+\epsilon$ for any $x'\in U_x$, $y'\in U_y$ and $t'\in U_t$. In particular, 
\[
f(x',t') - f(y',t') > f(x,t) - f(y,t) -2\varepsilon = 0.
\]
Thus $(x',t') \in \cU$ for any $(x',t')\in U_x\times U_t$, which implies that $\cU$ is open.

Finally, by assumption we know that $(W\times [0,1]) \setminus \cU$ is precisely the graph of the function $\sigma: [0,1]\to W$ where $\sigma(t)= v_t$. Since $\cU$ is open, the graph of $\sigma$ is closed. Thus $\sigma$ is continuous  by the closed graph theorem as $W$ is compact.
\end{proof}

\section{One dimensional dual complex}\label{s-dual complex dim 1}

\begin{say}
Let $x\in (X,\Delta)$ be a klt singularity, and $D$ a $\bQ$-complement such that $(X,\Delta+D)$ is klt outside $x$. For any collection of lc places $E_{t_1},\dots, E_{t_i}$ corresponding to points $t_1,\dots,t_i\in \cD(X,\Delta+D)$, by \cite[Corollary 1.4.3]{BCHM}, there exists a model 
\[
\mu_{t_1t_2\cdots t_i}\colon Y_{t_1t_2\cdots t_i}\to X
\]
which precisely extracts $E_{t_1},\dots, E_{t_i}$. Moreover, by running a $(\mu_{t_1t_2\cdots t_i*}^{-1}D)$-MMP over $X$, we may assume $-K_{Y_{t_1t_2\cdots t_i}}-\mu_{t_1t_2\cdots t_i*}(\Delta)-\sum^i_{j=1} E_{t_j}$ is nef, as the MMP sequence only has flips.

Such $Y_{t_1t_2\cdots t_i}$ is not unique, but any two models $Y_{t_1t_2\cdots t_i}$ and $Y'_{t_1t_2\cdots t_i}$ are crepant birationally equivalent. In particular, the notion of bigness of the restriction of $-K_{Y_{t_1t_2\cdots t_i}}-\mu_{t_1t_2\cdots t_i*}(\Delta)-\sum^i_{j=1} E_{t_j}$ on $E_{t_j}$ is well defined for any $1\le j\le i$.

Therefore, for any $t\in \cD(X,\Delta+D)$, $E_t$ is a \emph{Koll\'ar component}  if  $(Y_{t}, \mu_{t*}(\Delta)+E_t)$ is plt.; and $E_{t_1},\dots, E_{t_i}$ admits a Koll\'ar model, if there exists $Y_{t_1t_2\cdots t_i}$ such that the pair  $({Y_{t_1t_2\cdots t_i}}, \mu_{t_1t_2\cdots t_i*}(\Delta)+\sum^i_{j=1} E_{t_j})$ is qdlt and $-K_{Y_{t_1t_2\cdots t_i}}-\mu_{t_1t_2\cdots t_i*}(\Delta)-\sum^i_{j=1} E_{t_j}$ is ample. 

\end{say}
We will study the case that when $\cD(X,\Delta+D)$ is homeomorphic to a one dimensional interval $[0,1]$. For any $t\in [0,1]$, up to rescaling, it corresponds to a valuation $v_t$, and for $t\in \mathbb{Q}$, it corresponds to a divisorial valuation $E_t$.
\begin{lem}\label{lem-endkollar}
At least one of the endpoints corresponds to a Koll\'ar component.
\end{lem} 
\begin{proof}
Assume one ending point, say $E_1$, is not a Koll\'ar component. Then we know that we can construct a model $\mu_1\colon Y_1\to X$ which precisely extracts $E_1$. From our assumption,  $\mu_*^{-1}(D)$ does not contain the log canonical center  of $(Y_1, E_1+\mu_{1*}^{-1}(\Delta))$ properly contained in $E_1$. Therefore, 
$A_{{Y_1}, E_1+\mu_{1*}^{-1}(\Delta)}(E_0)=0$, so we can get a model $\mu_{10}\colon Y_{10}\to Y_1\to X$ extracting $E_0$, such that 
$-(K_{Y_{10}}+ E_0+E_1+\mu_{10*}^{-1}(\Delta))$ is nef, as it is the pull back of $-(K_{Y_{1}}+ E_1+\mu_{1*}^{-1}(\Delta))$.

Similarly, if $E_0$ is not a Koll\'ar component, we can get $\mu_{01}\colon Y_{01}\to X$ and $-(K_{Y_{01}}+ E_0+E_1+\mu_{01*}^{-1}(\Delta))$ is nef. In particular, $(Y_{10}, E_0+E_1+\mu_{10*}^{-1}(\Delta))$ and  $(Y_{01}, E_0+E_1+\mu_{01*}^{-1}(\Delta))$ are crepant birationally equivalent. 

However, the restriction of $-(K_{Y_{10}}+ E_0+E_1+\mu_{10*}^{-1}(\Delta))$ on $E_1$ is big and on $E_0$ is not big, while the restriction of $-(K_{Y_{01}}+ E_0+E_1+\mu_{01*}^{-1}(\Delta))$ on $E_0$ is big and on $E_1$ is not big. A contradiction. 
\end{proof}
By Lemma \ref{lem-endkollar}, we can always assume $E_0$ is a Koll\'ar component. 
Lemma \ref{lem-endkollar} also follows from the standard tie breaking argument, but the above proof sheds more light on our approach.  

\begin{prop}\label{prop-nghbKollar}
Assume $\cD(X,\Delta+D)$ is homeomorphic to $[0,1]$, and $E_0$, $E_1$ are Koll\'ar components. Then there exists a Koll\'ar model which precisely extracts $E_{0}$ and $E_{1}$.
\end{prop}
\begin{proof}
Let $\mu_{01}\colon Y_{01}\to X$ be a $\bQ$-factorial model which extracts $E_0$ and $E_1$ such that  $-K_{Y_{01}}-E_0-E_1-\mu_{01*}^{-1}\Delta$ is nef. In particular, $({Y_{01}}, E_0+E_1+\mu^{-1}_{01*}\Delta)$ is qdlt (see \cite[Proposition 34]{dFKX-dualcomplex}).
We claim  $(-K_{Y_{01}}-E_0-E_1-\mu^{-1}_{01*}\Delta)|_{E_i}$ is big for $i=0,1$. If not, say for $i=0$ this is not true, then let $f\colon Y_{01}\to Y_1'$ be the ample model of $-K_{Y_{01}}-E_0-E_1-\mu^{-1}_{01*}\Delta$  over $X$ which contracts $E_0$. Then $(Y_1', f_*(E_1+\mu^{-1}_{01*}\Delta))$ is not plt as it is crepant birationally equivalent to $({Y_{01}}, E_0+E_1+\mu^{-1}_{01*}\Delta)$. On the other hand, $Y_1'$ is isomorphic to $Y_1$ in codimension one. Thus by the negativity lemma, 
$({Y_1'},f_*(E_1+\mu^{-1}_{01*}\Delta))$  is crepant birationally equivalent to $({Y_1},E_1+\mu_{1*}^{-1}(\Delta))$,
which contradicts to the assumption that $E_1$ is a Koll\'ar component.

So if we take  the anti-canonical model $h\colon Z\to X$ of $Y_{01}$ over $X$ for $-K_{Y_{01}}-E_0-E_1-\mu^{-1}_{01*}\Delta$, it contains the birational transform $F_i$ of $E_i$ for $i=0,1$.

We claim a component $W$ of $F_0\cap F_1$ is of codimension two in $Z$. In fact, for sufficiently divisible $m$, the effective divisor $F=m(A_{X,\Delta}(F_0)F_0+A_{X,\Delta}(F_1)F_1)$ is Cartier and supported on $F_0\cup F_1$, hence $\cO_F$ is Cohen-Macaulay by \cite[Corollary 5.25]{KollarMori}.  If we localize at the generic point $\eta$ of $W$, as ${\rm Spec\,}\cO_{F,\eta}\setminus \{\eta\}$ is disconnected, $\dim({\rm Spec\,}\cO_{F,\eta})=1$ by \cite[Proposition 2.1]{Hartshorne-connectedness}.
By the classification of surface log canonical singularities (\cite[Section 3.3]{Kollar13}),
the pair $(Z,F_1+F_2+h_*^{-1}\Delta)$ is toriodal in a neighborhood $U$ of the generic point $\eta(W)$ of $W$, and $Y\to Z$ is isomorphic over $\eta(W)$, from the generic point of $E_0\cap E_1$. Then for any divisor $E$ exceptional over $Z$ whose center contained $Z\setminus U$, we have
\[
A_{Z,F_1+F_2+h_*^{-1}\Delta}(E)=A_{Y_{01} ,E_0+E_1+\mu^{-1}_{01*}\Delta}(E)>0\, .
\]
Thus $(Z,F_1+F_2+h_*^{-1}\Delta)$ is qdlt. In particular, it is a Koll\'ar model. 
\end{proof}

\begin{proof}[Proof of Theorem \ref{thm-nghbKollar}]
By Lemma \ref{lem-endkollar}, we may assume one of $E_0$ or $E_1$, say $E_0$, is a Koll\'ar component. By Proposition  \ref{prop-nghbKollar}, if $E_1$ is a Koll\'ar component, then Case (2) happens. So it remains to prove if $E_1$ is not a Koll\'ar component, then for any $s\in (0,1]\cap \bQ$, $E_s$ is not a Koll\'ar component.

Let $\mu_1\colon Y_1\to X$ be the model extracting $E_1$. From our assumption, $\mu_{1*}^{-1}D$ does not pass through the lc center of $(Y_1, E_1+\mu_{1*}^{-1}\Delta)$ properly contained in $E_1$. As a result,  $E_0$ and therefore any $E_s$ are lc places of $(Y_1, E_1+\mu_{1*}^{-1}\Delta)$. So we can extract $E_s$ over $Y_1$ to get a $\bQ$-factorial model $\mu_{1s}\colon Y_{1s}\to Y_1\to X$. Since 
\begin{eqnarray*}
\cD(X,\Delta+D) & \simeq &\cD(Y_1,E_1+\mu_{1*}^{-1}(\Delta+D)) \\ 
                       &\simeq  &\cD(Y_1,E_1+\mu_{1*}^{-1}\Delta)\simeq \cD(Y_{1s}, E_1+E_s+\mu_{1s*}^{-1}\Delta)\simeq[0,1]\, ,
\end{eqnarray*}
$W_1:=E_1\cap E_s$ is the only lc center properly contained in $E_1$ with $\dim(W_1)=\dim(X)-2$. In particular, $E_1$ and $W_1$ are normal and $W_1$ does not properly contain any lc center of $(Y_{1s}, E_1+E_s+\mu_{1s*}^{-1}\Delta)$. 
If we denote by $\nu\colon E_s^{\rm n}\to E_s$ the normalization, and write 
\[
(K_{Y_{1s}}+E_1+E_s+\mu_{1s*}^{-1}\Delta)|_{E_s^{\rm n}}=K_{E_s^{\rm n}}+\Delta_{E_s^{\rm n}} \, ,
\]
then $\cD(E_s^{\rm n}, \Delta_{E_s^{\rm n}})$ are two points corresponding to two disjoint lc centers $W_1'$ and $W_0'$. Since $(Y_{1s},E_1+E_s)$ is toroidal at the generic point $\eta(W_1)$ by the classification of surface log canonical singularities (see \cite[Section 3.3]{Kollar13}), $\nu\colon W_1'\to W_1$ is birational and $W_1\neq \nu(W'_0)$. Moreover, $W_0:=\nu(W'_0)$ does not meet $W_1$ as $W_1$ does not properly contain any lc center of $(Y_{1s}, E_1+E_s+\mu_{1s*}^{-1}\Delta)$. So $(Y_{1s}, E_1+E_s+\mu_{1s*}^{-1}\Delta)$ has two disjoint lc centers $W_0$ and $W_1$ properly contained in $E_s$.

We can run an $E_1$-minimal model program of $Y_{1s}$ over $X$, obtaining a birational model $Y_{1s}\dashrightarrow Y_s'$, which terminates by contracting $E_1$. This minimal model program is isomorphic outside $E_1$. Therefore, if we denote by $\mu_s'\colon Y'_{s}\to X$, then the divisorial part of ${\rm Ex}(\mu_s')$ is the birational transform $E'_{s}$ of $E_s$, and $(Y'_s,E'_{s}+{\mu_{s}'}_*^{-1}\Delta)$ has an lc center (isomorphic to $W_0$) properly contained in $E_s'$.

So if we extract $E_s$ to get a model $\mu_s\colon Y_s\to X$, such that $-(K_{Y_s}+E_{s}+\mu_{s*}^{-1}\Delta)$ is ample, then by the negativity lemma, $(Y_s, E_{s}+\mu_{s*}^{-1}\Delta)$ is not plt, i.e. $E_s$ is not a Koll\'ar component.
\end{proof}

The following result also illuminates the situation when $E_s$ is a Koll\'ar component for $s\in (0,1)$ by taking affine cone over a klt log Fano pair with complement containing only two lc places.

\begin{thm}\label{thm-CY-Gm}
Let $(X,\Delta)$ be a klt log Fano pair. Let $D$ be a $\bQ$-complement of $(X,\Delta)$ such that $\cD(X,\Delta+D)$ contains only two isolated points $E_1$ and $E_2$. Then $\Aut^0(X,\Delta+D)\cong \bG_m$, and both $E_1$ and $E_2$ induce product test configurations of $(X,\Delta+D)$ that are opposite to each other up to positive scaling.
\end{thm}

\begin{proof}
Let $Y\to X$ be the extraction of $E_1$ and $E_2$ such that each $E_i$ is anti-ample over $X$. 
Let $(\cX, \Delta_{\cX} +\cD_{\cX}) \to \bA^1_s$ be the weakly special test configuration induced by $E_1$ in the sense of boundary polarized CY pairs as in \cite{BLX-openness,ABB+}. Thus $\cX\to \bA^1$ is of Fano type. Let $\cE_i$ $(i=1,2)$ be the divisor over $\cX$ corresponding to $E_i\times \bA^1_s$. 
Let $\cY \to \cX$ be the extraction of $\cE_1$ and $\cE_2$ such that each $\cE_i$ is anti-ample over $\cX$. We claim that $(\cY, \Delta_{\cY} + \cD_{\cY} + \cY_0 + \cE_1 + \cE_2)$ is dlt. Clearly it is log canonical as it is crepant to $(\cX, \Delta_{\cX} +\cD_{\cX} + \cX_0)$ which is log canonical. In addition, it is plt away from $\cY_0$ as it is isomorphic to $(Y, \Delta_Y+ \cD_Y + E_1 + E_2)\times (\bA^1\setminus \{0\})$.  Denote by $\cE_{i,0}:= \cE_i|_{\cY_0}$. Then each $\cE_{i,0}$ is connected as the general fiber of $\cE_i\to \bA^1$ is connected. Since $(\cY_0, \Delta_{\cY_0}+\cD_{\cY_0} + \cE_{1,0} + \cE_{2,0})$ is crepant birational to $(\cX_0, \Delta_{\cX_0} + \cD_{\cX_0})$, we know that $\cE_{i,0}$ is reduced whose irreducible components are lc places of $(\cX_0, \Delta_{\cX_0} + \cD_{\cX_0})$. By \cite[Proposition 8.7]{ABB+}, we know that $\cD(\cX_0, \Delta_{\cX_0} + \cD_{\cX_0})$ has dimension $0$. Thus $\cE_{1,0}$ and $\cE_{2,0}$ are disjoint prime divisors. By \cite[Proposition 4.37]{Kollar13} we know that $\cE_{1,0}$ and $\cE_{2,0}$ are the only lc places of $(\cX_0, \Delta_{\cX_0} + \cD_{\cX_0})$. 
By inversion of adjunction, we know that $\cE_{1,0}$ and $\cE_{2,0}$ are the only minimal lc centers of $(\cY, \Delta_{\cY} + \cD_{\cY} + \cY_0 + \cE_1 + \cE_2)$. Since $\cY_0$ is regular at the generic point $\eta_{i,0}$ of $\cE_{i, 0}$, we have that $\cY$ is regular at $\eta_{i,0}$ as $\cY_0$ is Cartier in $\cY$. As a result, $(\cY, \Delta_{\cY}+\cD_{\cY} + \cY_0+ \cE_{1} + \cE_{2})$ is dlt at $\eta_{i,0}$ which implies that it is dlt everywhere. 

Next we show that $(\cX_0, \Delta_{\cX_0})$ is klt. From the above arguments we know that $\cE_{1,0}$ and $\cE_{2,0}$ are the only lc places of the slc pair $(\cX_0, \Delta_{\cX_0}+\cD_{\cX_0})$. Moreover, since $\ord_{E_i}(D)>0$, we have $\ord_{\cE_i}(\cD_{\cX})>0$ which implies that $\ord_{\cE_{i,0}}(\cD_{\cX_0})>0$. Thus $(\cX_0, \Delta_{\cX_0})$ is klt. 

So far we have shown that $E_1$ is a special divisor over $(X,\Delta)$. By symmetry so is $E_2$. Moreover, since $(\cX_0, \Delta_{\cX_0}+\cD_0)$ admits a $\bG_m$-action, we know that each $\cE_{i,0}$ induces a product test configuration of $(\cX_0, \Delta_{\cX_0}+\cD_0)$ by \cite{CZ22b} (see also \cite[Theorem 4.8]{ABB+}). In particular, $\cE_{2,0}$ is a special divisor over $(\cX_0, \Delta_{\cX_0})$. Thus $\cE_2$ provides a family of special divisors over $(\cX, \Delta_{\cX})$. By the proof of \cite[Proposition 4.5]{XZ-HRFG} (also see \cite[Theorem 5.7]{Xu-Kbook}), there exists a family of special test configurations $(\fX,\Delta_{\fX}) \to \bA^2_{s,t}$ of $\cX$ induced by $\cE_2$ where 
\[
(\fX,\Delta_{\fX})\times_{\bA^2}(\bA^2\setminus (t=0))\cong (\cX,\Delta_{\cX})\times (\bA^1_t\setminus\{0\})\,.
\] 
In particular, $(\fX_{(0,0)}, \Delta_{\fX_{(0,0)}})$ is a klt log Fano pair. Denote by $\fD_{\fX}$ the closure of $\cD\times  (\bA^1_t\setminus\{0\})$ in $\fX$. Then by \cite[Theorem 6.3]{ABB+} we know that $(\fX,\Delta_{\fX}+\fD_{\fX})\to \bA^2$ is a family of boundary polarized CY pairs. Since $\cE_2$ is $\bG_m$-equivariant for the natural $\bG_m$-action on $\cX$, we know that $\fX\to \bA^2_{s,t}$ is $\bG_m^2$-equivariant with the standard $\bG_m^2$-action on $\bA^2$. Moreover, the above arguments implies that the central fiber $(\fX_{(0,0)}, \Delta_{\fX_{(0,0)}}+\fD_{\fX_{(0,0)}})$ has only two lc places $\fE_{1,(0,0)}$ and $\fE_{2,(0,0)}$ which induce $1$-PS's of $\bG_m^2$ of weight $(1,0)$ and $(0,1)$ respectively. On the other hand, by \cite[Proof of Proposition 8.11]{ABB+} we know that 
\[
\Aut^0(\fX_{(0,0)}, \Delta_{\fX_{(0,0)}}+\fD_{\fX_{(0,0)}})\cong \bG_m\,.
\] 
As a result, we know that there exists a non-trivial $1$-PS $\sigma: \bG_m\to \bG_m^2$ of weight $(a,b)$ such that $\sigma$ acts trivially on $\fX_{(0,0)}$. Since the $1$-PS's of weight $(1,0)$ and $(0,1)$ are induced by valuations with different centers, we know that $a,b$ have the same sign, and we may assume $a>0$ and $b>0$. Thus by pulling back $\fX\to \bA^2$ under $\sigma$ we obtain a test configuration of $(X, \Delta+D)$ whose $\bG_m$-action on the central fiber $(\fX_{(0,0)}, \Delta_{\fX_{(0,0)}}+\fD_{\fX_{(0,0)}})$ is trivial. Therefore, we must have 
\[
(X, \Delta+D)\cong (\fX_{(0,0)}, \Delta_{\fX_{(0,0)}}+\fD_{\fX_{(0,0)}})
\]
which implies that $\Aut^0(X, \Delta+D)\cong \bG_m$, and the proof is finished by \cite{CZ22b}.
\end{proof}

\begin{expl}\label{ex-conecubic}
Let $V=\mathbb P^2=\bP(x,y,z)$ and $D_V$ be the nodal cubic $D_V=(zx^2+zy^2+y^3=0)$. The lc places of $(V,D_V)$ is a cone over a circle.  Let $u_t$ $(0<t<+\infty)$ be the quasi-monomial valuations with weight $(1,t)$ over the two branches of $D_V$ at $[0:0:1]$. Then \cite[Section 6]{LXZ-HRFG} shows that $u_t$ is a special valuation, i.e. ${\rm Proj}\left({\rm Gr}_{u_t} k[x,y,z]\right)$ is a klt Fano variety if and only if 
\[
t\in \left(\frac{7-3\sqrt{5}}{2},\frac{7+3\sqrt{5}}{2}\right)\, .
\]

Now we consider the affine cone $(\bA^3,D)$ of $(V,D_V)$ with polarization $\cO_V(1)$, so $o\in \bA^3$ is the origin, and $D$ is the divisor which is the cone over $D_V$ whose affine equation is given by $D=(zx^2+zy^2+y^3=0)\subset \bA^3$. 
Then the dense open subcomplex $\cD(\bA^3,D)^\circ \subset\cD( \bA^3,D)$ consisting of valuations centered at $o$ is of the form $v_{s,t}=(\ord_{o},s\cdot u_{\frac{t}{s}})$, where  $s,t\in [0,+\infty)$ with $v_{s,0}$ and $v_{0,s}$ glued (corresponding to $(\ord_o, s\cdot \ord_D)$). By writing $v = (\ord_{o}, u)$ for $u\in \Val_{V}$ we mean that for $g = \sum_{m\geq 0} g_m \in \cO_{\bA^3, o}$ with $g_m$ the homogeneous component of degree $m$ of $g$, we have 
\[
v(g) : = \min \{m + u(g_m) \mid g_m\neq 0\}.
\]
From this expression, we know that $v_{s,t}$ is a quasi-monomial valuations with weights $(1,s,t)$ in the blow-up of $\bA^3$ at $o$ with respect to the exceptional divisor and the two branches of the strict transform  of $D$.
Let $\bm$ be the maximal ideal corresponding to $o$. So for $\fa=(f,\bm^{\ell})$ with $f = zx^2+zy^2+y^3$ and $\ell \in \bN$, we have
\[
v_{s,t}(\fa)= \begin{cases}
v_{s,t}(f)=3+s+t &\text{if  $3+s+t \le \ell$}\\
\ell  &\text{if $3+s+t\ge \ell$}
\end{cases}
\]
and the log discrepancy $A_{\bA^3}(v_{s,t})=3+s+t $. For any fixed $\ell >3$, we choose a set of general generators $h_1, h_2, \cdots, h_m\in \fa$, and let $D':=\frac{1}{m}(H_1+H_2+\cdots +H_m)$ with $H_i := (h_i=0)$. Then the dual complex
\[
\cD(\bA^3, D') = \{v_{s,t}\in \cD(\bA^3, D)^\circ \mid s,t\geq 0\textrm{ and }s+t \leq \ell -3\}.
\]
Moreover, the set of Koll\'ar valuations  $\cD^{\rm KV}(\bA^3, D')$ consist of $\ord_o = v_{0,0}$ and $v_{s,t}$ with
\[
 s,t>0, \quad \frac{t}{s}\in \left(\frac{7-3\sqrt{5}}{2},\frac{7+3\sqrt{5}}{2}\right) \mbox{\ \ and \ \ } s+t\le \ell-3\, .
\]
As a result, we see that $\cD^{\rm KV}(\bA^3, D')$ is neither open in $\cD(\bA^3, D')$, nor  a finite union of open simplicies in any rational triangulation of $\cD(\bA^3, D')$. On the other hand, it is not hard to see that Conjecture \ref{conj-kollar valuations} holds in this example.
\end{expl}

\bibliography{ref}

\end{document}